\documentclass[reqno,tbtags]{amsproc}

\usepackage{amsthm}

\newtheorem{theorem}{Theorem}
\newtheorem{proposition}{proposition}
\theoremstyle{definition}
\newtheorem{definition}{Definition}
\newtheorem{remark}{Remark}

\begin{document}

\title{Relative index theorem in $K$-homology}

\author{V.~E.~Nazaikinskii}
\address{A.~Ishlinsky Institute for Problems in Mechanics,
Moscow; Moscow Institute of Physics and Technology, Dolgoprudny,
Moscow District} \email{nazaikinskii@yandex.ru}


\maketitle

\begin{abstract}
We prove an analog of Gromov--Lawson type relative index theorems
for $K$-homology classes.
\end{abstract}

\section*{Introduction}

Let $M$ and~$M'$ be two manifolds coinciding outside some subsets
$Q\subset M$ and $Q'\subset M'$ (i.e., $M\setminus Q$ and
$M'\setminus Q'$ are identified with each other), and let
$D$~and~$D'$ be elliptic operators on~$M$ and~$M'$, respectively,
coinciding on $M\setminus Q\simeq M'\setminus Q'$. The difference
$\operatorname{ind}D-\operatorname{ind}D'$ of their indices is
called the \textit{relative index} of~$D$ and~$D'$. A
\textit{relative index theorem} is a statement of the following
type: \textit{the relative index is independent of the structure
of~$M$ and~$M'$, as well as of~$D$ and~$D'$, on the set where they
coincide, i.e., on $M\setminus Q$}; in other words, to compute the
relative index, it suffices to know the structure of~$D$ and~$D'$
on~$Q$ and~$Q'$, respectively. Such theorems are trivial for smooth
closed manifolds (owing to the existence of a local index formula;
e.g., see~\cite{Gil1}), but they are informative in more general
cases. For example, a relative index theorem for Dirac operators on
complete noncompact Riemannian manifolds was proved in the famous
paper~\cite{GrLa1} by Gromov and Lawson. For further examples, we
refer the reader to the paper~\cite{R:NaSt7}, where the relative
index theorem was proved in a rather general abstract framework
that not only included many of the earlier known theorems as
special cases but also permitted one to obtain a number of index
formulas for elliptic differential operators and Fourier integral
operators on manifolds with singularities (see~\cite{NSScS99}).
Note, however, that index is not the only homotopy invariant of
elliptic operators, and hence it is of interest to obtain locality
theorems for broader sets of invariants. There are various
directions in which to generalize the relative index theorem. For
example, Bunke~\cite{Bun95} considered Dirac operators acting on
sections of projective bundles of Hilbert $B$\nobreakdash-modules
over a complete noncompact Riemannian manifold, where $B$~is a
$C^*$\nobreakdash-algebra, and obtained a relative index theorem
for such operators, the index being an element of the $K$-group
of~$B$. Here we solve a different problem. If the elliptic
operators in question are local with respect to some
$C^*$-algebra~$A$, then it is natural to ask how the corresponding
classes in the $K$-homology of~$A$ vary under a ``local'' variation
of the operator. Here the algebra~$A$ is not assumed to be
commutative, and accordingly, localization is based on ideals
in~$A$. It turns out---which is the main result of the paper---that
this variation obeys the same laws as the relative index in the
``classical'' theorems does. That is why we still refer to our
theorem as a ``relative index theorem,'' even though it deals with
$K$-homology classes rather than the index. All results are stated
in terms of Fredholm modules; for the standard construction that
assigns a Fredholm module to an elliptic operator, we refer the
reader to the literature (e.g., see~\cite{HiRo1}).

\section{$K$-homology}

Recall the definition of $K$-homology groups of a $C^*$-algebra~$A$
(see \cite[Chap.~8]{HiRo1}). A \textit{Fredholm module} over~$A$ is
a triple $x=(\rho,H,F)$, where $H$ is a Hilbert space, $\rho\colon
A\to\mathfrak{B}(H)$ is a representation of~$A$ on~$H$, and
$F\in\mathfrak{B}(H)$ is an operator such that
\begin{equation}\label{nazeq-1}
    [F,\rho(\varphi)]\sim 0\quad\forall\varphi\in A\quad
    \text{(\textit{locality})},\qquad  F\approx F^*,\qquad
    F^2\approx 1,
\end{equation}
where $\sim$ stands for equality modulo compact operators and
$\approx$ for equality modulo \textit{locally compact operators},
i.e., operators~$C$ such that the operators $\rho(\varphi)C$ and
$C\rho(\varphi)$ are compact for every $\varphi\in A$. Two Fredholm
modules $(\rho,H,F_0)$ and $(\rho,H,F_1)$ corresponding to one and
the same representation~$\rho$ are said to be \textit{homotopic} if
they can be embedded in a family $(\rho,H,F_t)$, $t\in[0,1]$, of
Fredholm modules such that the function $t\mapsto F_t$ is operator
norm continuous. A Fredholm module is said to be
\textit{degenerate} if all relations in~\eqref{nazeq-1} are
satisfied exactly rather than modulo (locally) compact operators.
We say that two Fredholm modules~$x$ and~$x'$ are
\textit{equivalent} if there exists a degenerate module~$x''$ such
that the modules $x\oplus x''$ and $x'\oplus x''$ are unitarily
equivalent to homotopic Fredholm modules. The set of equivalence
classes of Fredholm modules is denoted by $K^1(A)$; the direct sum
of modules induces a structure of an abelian group on $K^1(A)$,
which is called the (\textit{odd}) \textit{$K$-homology group
of}~$A$. The definition of the \textit{even $K$-homology
group}~$K^0(A)$ is completely similar; here one considers
\textit{graded} Fredholm modules, i.e., ones equipped with the
following additional structure: the space~$H$ is
$\mathbb{Z}_2$-graded, $H=H_+\oplus H_-$, the representation~$\rho$
is even (i.e., preserves the grading, $\rho(A)H_\pm\subset H_\pm$),
and the operator~$F$ is odd (i.e., $FH_+\subset H_-$ and
$FH_-\subset H_+$).

The results stated below hold for $K^0(A)$ as well as $K^1(A)$, and
it is tacitly assumed throughout that all Fredholm modules involved
are graded in the first case. For brevity, we often write $\varphi$
rather than $\rho(\varphi)$; which representation is meant is
always clear from the context.

\section{Fredholm modules agreeing on an ideal}

Let $x=(\rho,H,F)$ and $\widetilde x=(\widetilde\rho,\widetilde
H,\widetilde F)$ be Fredholm modules over $A$, and let $J\subset A$
be an ideal. The orthogonal projections\footnote{The subspace~$JH$,
as well as the subspaces~$J_1H$ and~$J_2H$ considered below, is
closed. This is a special case of the general assertion that the
subspace $BH$ of a Hilbert space~$H$ equipped with a representation
of a $C^*$\nobreakdash-algebra~$B$ is closed (see~\cite[pp.~25--26,
Sec.~1.9.17]{HiRo1}).} $P\colon H\to H_0$, where $H_0=JH\subset H$,
and $\widetilde P\colon \widetilde H\to \widetilde H_0$, where
$\widetilde H_0=J\widetilde H$, commute with the action of~$A$.
\begin{definition}
Given an operator $T\colon H_0\to\widetilde H_0$ intertwining the
representations~$\rho$ and~$\widetilde\rho$, preserving the grading
in the graded case, and satisfying $TPFPT^*\approx\widetilde
P\widetilde F\widetilde P$, we say that $x$~and~$\widetilde x$
\textit{agree on the ideal}~$J$.
\end{definition}

\section{Cutting and pasting}\label{3}

Let  $J_1,J_2\subset A$ be ideals such that $J_1+J_2=A$. Let
$x$~and~$\widetilde x$ be Fredholm modules over~$A$ agreeing on the
ideal $J=J_1\cap J_2$, and assume that the representations~$\rho$
and~$\widetilde\rho$ are nondegenerate (i.e., $AH=H$ and
$A\widetilde H=\widetilde H$).  Then one can define a Fredholm
module $x\diamond\widetilde x$ obtained, informally speaking, by
``pasting together over~$J$ the part of~$x$ corresponding to~$J_1$
with the part of~$\widetilde x$ corresponding to~$J_2$.'' To this
end, we represent~$F$ (and, in a similar way, $\widetilde F$) by
a~$3\times3$~matrix associated with the decomposition of~$H$ into
the direct orthogonal sum of the $A$-invariant subspaces $H_0=JH$,
$H_1=J_1H\ominus H_0$ (the orthogonal complement), and
$H_2=J_2H\ominus H_0$, $H=H_1\oplus H_0\oplus H_2$ (in this
particular order!).\footnote{In specific examples, some of the
subspaces~$H_0$, $H_1$, and~$H_2$ may prove to be trivial (zero).
Our argument remains valid in this case, but the result is not of
much interest.} We denote the orthogonal projection onto~$H_j$
by~$P_j$, $j=0,1,2$. Note that\footnote{From now on, for an
arbitrary $\varphi\in A$, by $\varphi_1\in J_1$ and $\varphi_2\in
J_2$ we denote arbitrary elements such that
$\varphi=\varphi_1+\varphi_2$.} $\varphi
P_1FP_2=\varphi_1P_1FP_2=P_1\varphi_1FP_2\sim P_1F\varphi_1P_2=0$
for any $\varphi\in A$; i.e., $P_1FP_2\approx 0$, and likewise
$P_2FP_1\approx0$, so that the desired representation can be
written out as
\begin{equation}\label{nazeq-2}
    F\approx
\begin{pmatrix}
      a & b & 0 \\
      b^* & c & d \\
      0 & d^* & e \\
    \end{pmatrix},\qquad \widetilde F\approx
\begin{pmatrix}
      \widetilde a & \widetilde b & 0 \\
      \widetilde b^* & \widetilde c & \widetilde d \\
      0 & \widetilde d^* & \widetilde e \\
    \end{pmatrix},\qquad
    \begin{matrix}
      a=a^*, & c=c^*, & e=e^*, \\
      \widetilde a=\widetilde a^*, & \widetilde c=\widetilde c^*, &
      \widetilde e=\widetilde e^*, \\
    \end{matrix}
\end{equation}
where all entries are local. The condition that~$x$ and~$\widetilde
x$ agree on~$J$ acquires the form $TcT^*\approx\widetilde c$. To
simplify the notation, we identify~$H_0$ with~$\widetilde H_0$
via~$T$; then we no longer write out~$T$ explicitly, and the
agreement condition on~$J$ becomes $c\approx\widetilde c$. Set
\begin{equation}\label{nazeq-3}
\begin{gathered}
    H\diamond\widetilde H=H_1\oplus H_0\oplus\widetilde H_2,\;\;
    \rho\diamond\widetilde\rho=\rho|_{H_1\oplus
    H_0}\oplus\widetilde\rho|_{\widetilde H_2},\;\;
    F\diamond\widetilde F= \begin{pmatrix}
      a & b & 0 \\
      b^* & c &\widetilde d \\
      0 & \widetilde d^* &\widetilde e \\
    \end{pmatrix}.
\end{gathered}
\end{equation}

\begin{proposition}
The Fredholm module $x\diamond\widetilde
x=(\rho\diamond\widetilde\rho,H\diamond\widetilde
H,F\diamond\widetilde F)$ over~$A$ is well defined by formulas
\eqref{nazeq-3}.
\end{proposition}
\begin{proof}
In terms of the matrix in~\eqref{nazeq-2}, the condition
$F^2\approx 1$ becomes\footnote{Here $1$ stands for the identity
operators on relevant subspaces.}
\begin{gather}\label{e1}
    a^2+bb^*\approx 1,\quad ab+bc\approx 0,\quad cd+de\approx 0,\quad
    d^*d+e^2\approx 1,\qquad bd\approx0,
    \\\label{e3}
    \varphi b^*b\sim\varphi_1(1-c^2),\quad \varphi dd^*\sim\varphi_2(1-c^2)
    \qquad\forall \varphi\in A,
\end{gather}
and the last condition in~\eqref{e1} is satisfied automatically
($\varphi bd=(\varphi_1b)d\sim (b\varphi_1)d=b(\varphi_1d)=0$),
while condition~\eqref{e3} follows from the fact that
$b^*b+dd^*+c^2\approx 1$. Similar relations hold for $\widetilde
F$. To prove the proposition, it suffices to verify that
$(F\diamond\widetilde F)^2\sim1$. (The other conditions in
\eqref{nazeq-1} obviously hold for $x\diamond\widetilde x$.) The
verification, after squaring the matrix, is reduced to routine
calculations using the relation $c\approx\widetilde c$ and also
relations \eqref{e1}--\eqref{e3} for~$F$ and~$\widetilde F$. For
example, for the entry in the second line and second row, we obtain
$\varphi ((F\diamond\widetilde F)^2)_{22}=
\varphi(b^*b+c^2+\widetilde d\widetilde
d^*)\sim\varphi_1(1-c^2)+\varphi c^2+\varphi_2(1-c^2)=\varphi 1$,
$\varphi\in A$.
\end{proof}

The Fredholm module $\widetilde x\diamond x$ is defined in a
similar way.

\section{Relative index theorem}

Now we are in a position to state our main result.
\begin{theorem}
Under the assumptions of Sec.~\textup{\ref{3}}, one has
\begin{equation}\label{e7}
  [x\diamond\widetilde x]-[x]=[\widetilde x]-[\widetilde x\diamond x],
\end{equation}
where $[y]\in K^*(A)$ is the element defined by a Fredholm
module~$y$.
\end{theorem}
Identity~\eqref{e7} means that the difference of $K$-homology
classes resulting from the nonagreement of Fredholm modules over
the ideal~$J_2$ is independent of the structure of these modules
over the ideal~$J_1$ (where they agree).

\begin{remark}
As far as the author is aware, the result is new even for a
commutative algebra~$A$ in that relation~\eqref{e7} is established
for elements of the $K$-homology group rather than for the indices
of the operators in question. (Note, however, that this was
essentially done ``behind the scenes'' in~\cite{Bun95} for the case
in which $A$ is a function algebra on a complete noncompact
Riemannian manifold and the Fredholm modules correspond to some
Dirac type operators.) The classical relative index theorems can be
obtained from our result if one assumes that $A$~is a unital
function algebra: it suffices to use the homomorphism
$\operatorname{ind}\colon K^0(A)\longrightarrow
K^0(\mathbb{C})\simeq\mathbb{Z}$ corresponding to the natural
embedding of~$\mathbb{C}$ in~$A$. Thus, the theorem stated above
can be viewed as a natural generalization of relative index
theorems in the framework of noncommutative geometry.

Note also that a similar theorem holds in Kasparov's $KK$-theory.
It is considered in a separate paper~\cite{RJMP}. The reason for
separate analysis is that although these two theorems do overlap,
they are not corollaries of one another. Namely, both theorems deal
with elements of $KK(A,B)$, but
\begin{itemize}
    \item (This paper) $B=\mathbf{C}$, and $A$ is arbitrary.
    \item (\cite{RJMP}) $B$ is arbitrary, and $A$ is unital.
\end{itemize}
Moreover, even the construction of the module $F\diamond\widetilde
F$ in~\cite{RJMP} is different: in the present paper, we use
projections, but in~\cite{RJMP} we are forced to use a partition of
unity in~$A$, because appropriate projections do not necessarily
exist in Kasparov $(A,B)$-modules.

\end{remark}

\begin{proof}[Outline of the proof]
It suffices to deform the Fredholm module $x\oplus\widetilde x$ to
a module that is unitarily equivalent to the module
$(x\diamond\widetilde x)\oplus(\widetilde x\diamond x)$. The
homotopy is given by the family of Fredholm modules
$(\rho\oplus\widetilde\rho,H\oplus\widetilde H,\mathcal{F}_t)$,
$t\in[0,\pi/2]$, where the operator $\mathcal{F}_t$ is specified in
the direct sum decomposition $H\oplus\widetilde H=H_1\oplus
H_0\oplus H_2\oplus\widetilde H_1\oplus H_0\oplus\widetilde H_2$ by
the $6\times 6$ block matrix
\begin{equation}\label{e8}
    \mathcal{F}_t=\begin{pmatrix}
      a & b & 0 & 0 & 0 & 0 \\
      b^* & c & d\cos t & 0 & 0 & -\widetilde d\sin t \\
      0 & d^*\cos t & e & 0 & d^*\sin t & 0 \\
      0 & 0 & 0 & \widetilde a & \widetilde b & 0 \\
      0 & 0 & d\sin t & \widetilde b^* & c & \widetilde d\cos t \\
      0 & -\widetilde d^*\sin t & 0 & 0 & \widetilde d^*\cos t & \widetilde e \\
    \end{pmatrix}.
\end{equation}
The first and second conditions in \eqref{nazeq-1} are obvious for
$\mathcal{F}_t$, and the third condition ($\mathcal{F}_t^2=1$) can
be verified by routine computations. Next,
$\mathcal{F}_0=F\oplus\widetilde F$ and
\begin{equation}\label{e9}
    \mathcal{F}_{\pi/2}=\begin{pmatrix}
      a   & b       & 0 & 0       & 0     & 0     \\
      b^* & c       & 0 & 0       & 0     & -\widetilde d\\
      0   & 0       & e & 0       & d^*   & 0     \\
      0   & 0       & 0 & \widetilde a   & \widetilde b & 0     \\
      0   & 0       & d & \widetilde b^* & c     & 0     \\
      0   & -\widetilde d^*& 0 & 0 & 0& \widetilde e            \\
    \end{pmatrix}=U^*\begin{pmatrix}
      a   & b       & 0     & 0       & 0     & 0 \\
      b^* & c       &  \widetilde d& 0       & 0     & 0 \\
      0   &  \widetilde d^*& \widetilde e & 0       & 0     & 0 \\
      0   & 0       & 0     & \widetilde a   & \widetilde b & 0 \\
      0   & 0       & 0     & \widetilde b^* & c     & d \\
      0   & 0       & 0     & 0       & d^*   & e \\
    \end{pmatrix}U,
\end{equation}
where the unitary operator
\begin{multline*}
    U\colon H\oplus\widetilde H\equiv H_1\oplus
H_0\oplus H_2\oplus\widetilde H_1\oplus H_0\oplus\widetilde H_2\\
\longrightarrow H_1\oplus H_0\oplus\widetilde H_2\oplus\widetilde
H_1\oplus H_0\oplus H_2\equiv (H\diamond \widetilde H)\oplus
(\widetilde H\diamond H)
\end{multline*}
interchanges the third and sixth components and then multiplies the
third component by $-1$. Thus,
$\mathcal{F}_{\pi/2}=U^*((F\diamond\widetilde F)\oplus(\widetilde
F\diamond F))U$, as desired.
\end{proof}

\subsection*{Acknowledgements}

I am grateful to A.~Yu.~Savin and B.~Yu.~Sternin for useful
discussion.

\end{document}